\documentclass[10pt,a4paper]{amsart}
\usepackage{amsmath}
\usepackage{amssymb}
\usepackage{amsopn}
\usepackage{amsfonts}
\usepackage{latexsym}
\usepackage{ifpdf}
	\ifpdf
	\usepackage[pdftex]{graphicx}
	\else
	\usepackage{graphicx}
	\fi

\theoremstyle{plain}
\newtheorem{Thm}{Theorem}
\newtheorem{Lem}{Lemma}
\newtheorem{Prop}{Proposition}
\newtheorem{Cor}{Corollary}
\theoremstyle{definition}
\newtheorem{Def}{Definition}
\newtheorem{Rmk}{Remark}

\newcommand{\Z}{\mathbb Z}

\newcommand{\N}{\mathbb N}

\author[Cornelis {\sc Kraaikamp}]{{\sc Cornelis} KRAAIKAMP}
\address{Cornelis {\sc Kraaikamp}\\
EWI, Delft University of Technology \\
Mekelweg 4, 2628 CD Delft \\
the Netherlands }
\email{c.kraaikamp@tudelft.nl}

\author[Thomas A. {\sc Schmidt}]{{\sc Thomas A.}
SCHMIDT}
\address{Thomas A. {\sc Schmidt}\\
Department of Mathematics\\
Oregon State University\\
Corvallis, OR 97331\\
USA}
\email{toms@math.orst.edu}

\author[Ionica {\sc Smeets}]{{\sc Ionica} SMEETS}
\address{Ionica {\sc Smeets}\\
Mathematical Institute\\
Leiden University\\
Niels Bohrweg 1, 2333 CA Leiden\\
 the Netherlands}
\email{ionica.smeets@gmail.com}
\thanks{The second author was supported in part by Bezoekersbeurs \textbf{B 040.11.083} of
the Nederlandse Organisatie voor Wetenschappelijk Onderzoek (NWO) and also thanks TU Delft for providing a stimulating atmosphere.} 
\keywords{Rosen fractions, natural extensions, entropy}
\subjclass[2000]{11K50, 37A45,   37A35}
\date{11 March, 2009}

\title[Quilting natural extensions for $\alpha-$ Rosen  Fractions]{Natural extensions for  \boldmath{$\alpha-$}Rosen Continued Fractions}

\usepackage{hyperref}

\begin{document}

\begin{abstract}
We give natural extensions for the $\alpha$-Rosen continued fractions of Dajani {\em et al}  for a set of small $\alpha$ values by appropriately adding and deleting rectangles from the region of the natural extension for the standard Rosen fractions.   It follows that  the underlying maps have equal entropy.
\end{abstract}

\maketitle

\section{Introduction}  
Dajani, Kraaikamp and Steiner \cite{DKS} introduced  a generalization of both the $\alpha$-continued fractions of Nakada \cite{N} and the Rosen continued fractions \cite{R}.   Using direct methods, similar to those of \cite{BKS} for the classical Rosen fractions, they determine in \cite{DKS} natural extensions for certain of their $\alpha$-Rosen fractions.

  Here we give a method that begins with the explicit region of the natural extension of a Rosen fraction (as already determined in \cite{BKS}) and determines the regions for the natural extensions for 
various $\alpha$-Rosen fractions.     One advantage of this approach is that one easily sees that these various $\alpha$-Rosen fraction maps determine isomorphic dynamical systems;  in particular the associated one-dimensional maps have the same entropy.     This can be compared to results on the entropy of Nakada's $\alpha$-continued fractions, obtained by Nakada \cite{N} and others \cite{MCM}, \cite{LM}, \cite{NN}.   

We define all notation below, here we simply note that  for each $q\in \Z, q \geq 3$ and all appropriate values of $\alpha$, one can define an  $\alpha$-Rosen continued fraction map $T_{\alpha}$ on $[\, \lambda(\alpha-1),\lambda\alpha\,)\,$.       When $\alpha= 1/2$, the map is the original map of Rosen;   \cite{DKS} 
showed that the domain of the natural extension is connected for all   $\alpha \in [\,1/2, 1/\lambda\,]\,$.    We determine the least $\alpha_0$ such that  for all $\alpha \in (\,\alpha_0, 1/\lambda\,]\,$ the   natural extension is connected.      

 We prove the following.
\begin{Thm}\label{thmAnnounce}  Fix $q\in \Z, q \geq 4$ and $\lambda = \lambda_{q}= 2 \cos \frac{\pi}{q}$. \begin{itemize}
\item[(i.)] Let  
\[ \alpha_0 := \begin{cases}    \dfrac{\lambda^2-4+\sqrt{(4-\lambda^2)^2+4\lambda^2}}{2\lambda^2} &\text{if}\; q \; \text {is even},\\
\\
                               \dfrac{\lambda -2 +\sqrt{2\lambda^2-4\lambda+4}}{\lambda^2} &\text{otherwise}.
                                \end{cases}
\]      
Then $\left(\, \alpha_0, 1/\lambda \,\right]$ is the largest interval containing $1/2$ for which each 
domain of the natural extension of $T_{\alpha}$  is connected.     

\item[(ii.)]  Furthermore,  let 
\[ \omega_0 := \begin{cases}    1/\lambda &\text{if}\; q \; \text {is even},\\
                               \dfrac{\lambda-2 + \sqrt{\lambda^2 - 4 \lambda + 8}}{2 \lambda}&\text{otherwise}.
                                \end{cases}
\]   
Then  the entropy of the $\alpha$-Rosen map for each $\alpha \in \left[\, \alpha_0, \omega_0 \,\right]\,$ is equal to the entropy of the standard Rosen map. 
\end{itemize}
\end{Thm}
(Here, ``the domain of the natural extension of $T_{\alpha}$'' refers  to the largest region on which  the standard number theoretic map --- defined below in Equation ~\eqref{def: T(x,y)_alpha} --- is bijective.)
The value of the entropy of the standard Rosen map was given by H.~Nakada \cite{N2}, and is equal to 
\[  C \cdot \dfrac{(q-2) \pi^2}{2 q}\;,\]
where $C$ is the normalizing constant (which depends on the parity of the index $q$) found by \cite{BKS}.  We recall the value of $C\,$  in Equations \eqref{evenNormConst} and \eqref{oddNormConst}.

\begin{figure}[!ht]
\scalebox{0.9}{\includegraphics[height=40mm]{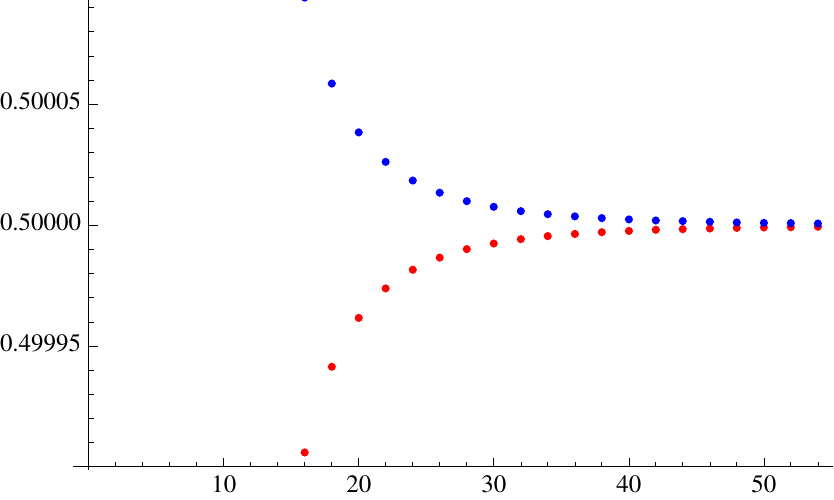}}
\scalebox{0.9}{\includegraphics[height=40mm]{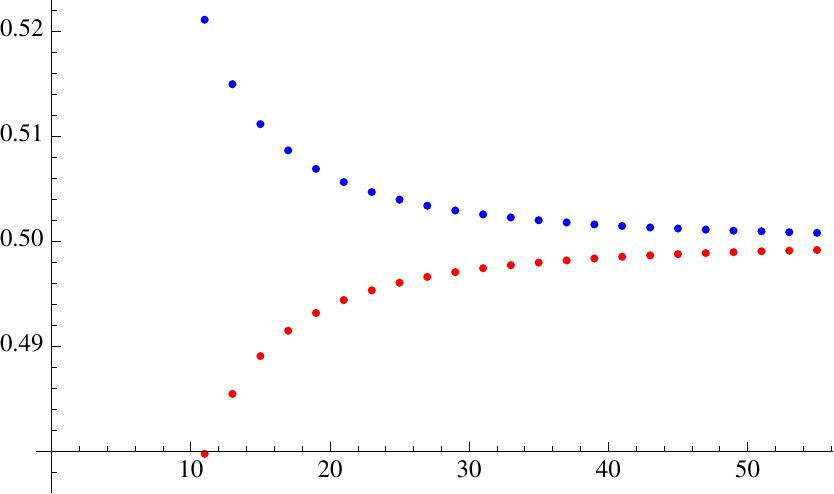}}
\caption{Some values of $\alpha_0$ and of $\omega_0\,$.  Even index on left, odd index on right.}
\label{alphaBdsFig}
\end{figure}

Our approach is to determine the domain of the natural extension using a standard ``number theoretic'' two-dimensional operator, here  $\mathcal T_{\alpha}\,$ (depending on both $\lambda$ and $\alpha$).   The domain for Rosen fractions, denoted here by $\Omega_{1/2}\,$, was obtained in   \cite{BKS}.       For fixed $\lambda$ and given $\alpha$, we determine our domain by adding and deleting various regions to $ \Omega_{1/2}$ by a process that we informally refer to as ``quilting''.     The regions are determined from appropriate orbits of the strips in $\Omega_{1/2}$ above intervals where the $T_{\alpha}$-``digits'' (see below) differ from the $T_{1/2}$-digits.    These strips are mapped by $\mathcal T_{1/2}$ to a region that must be deleted, their image under  $\mathcal T_{\alpha}$ must be added;   thereafter their further  $\mathcal T_{\alpha}$-orbits are deleted and added, respectively.        Our approach succeeds without great difficulty exactly because in a small number of steps the orbits of the added regions agree with the orbits of the deleted regions  --- the infinitely many potential holes are quilted over by the added regions.

\subsection{Outline of Paper}   The following two subsections complete this Introduction by giving basic notation and then defining our building blocks, the basic deleted and added regions.   In Section ~\ref{secQuiltGood} we sketch a main argument of our approach --- when the orbits of these basic regions agree after the same number of steps,  entropy is preserved.  We give an example of our techniques in  Section ~\ref{classCaseSec} by re-establishing known results for certain classical Nakada $\alpha$-fractions.    In Sections ~\ref{SecEven} and ~\ref{SecOdd} we give the proof of Theorem \ref{thmAnnounce}, in the even and odd index case, respectively.    Finally, in Section ~\ref{secDKSlargeAlpha}  we indicate how our results can be extended to  show that in the odd index case,  the entropy of $T_{\alpha}$ decreases when  $\alpha > \omega_0\,$.

\subsection{Basic Notation} 
\subsubsection{One dimensional maps}
Let $q\in \Z, q \geq 3$ and $\lambda = \lambda_{q}=\cos \frac{\pi}{q}$. For $\alpha \in \left[\,0, \frac{1}{\lambda}\,\right],$ we let $\mathbb I_{\alpha} : = [\, \lambda(\alpha-1),\lambda\alpha\,)\,$ and define the map $T_{\alpha}:\mathbb I_{\alpha} \to \mathbb I_{\alpha}$ by
\begin{equation}
\label{def: Talpha}
T_{\alpha}(x) := \left| \frac{1}{x} \right| - \lambda \left \lfloor\,  \left| \frac{1}{\lambda x} \right| + 1 -\alpha \right\rfloor, \textrm{ for } x\neq 0; \, T_{\alpha}(0):=0.
\end{equation}
For  $x \in\mathbb I_{\alpha}\,$,  put $d(x) := d_{\alpha}(x) = \left\lfloor \left| \frac{1}{\lambda x} \right| + 1 -\alpha \right\rfloor$ and  
$\varepsilon(x) :=\rm{sgn}(x)$. Furthermore, for $n \geq 1$ with $T_{\alpha}^{n-1}(x) \neq 0$ put
\[
\varepsilon_{n}(x)=\varepsilon_n = \varepsilon(T^{n-1}_\alpha (x)) \textrm{ and } d_n(x)=d_n=d(T^{n-1}_\alpha (x)).
\]
This yields the {\em $\alpha$-Rosen continued fraction}  of $x\,$:  
\[
x = \displaystyle{\frac{\varepsilon_1}{d_1\lambda +  \displaystyle{\frac{\varepsilon_2}{d_2\lambda + \dots}}}} =: [\varepsilon_1:d_1,\varepsilon_2:d_2,\dots],
\]
where $\varepsilon \in \{ \pm 1\}$ and $d_i \in \N$.    Fixing $\alpha = \frac{1}{2}\,$,  results in the Rosen fractions.   On the other hand, fixing  $q=3\,$ and considering general $\alpha\,$, we have Nakada's $\alpha$-expansions.   These include   the regular continued fractions, given by  $\alpha=\frac1\lambda=1\,$   and the nearest integer continued fractions, given  by  $\alpha=\frac12\,$.\\

\subsubsection{Dynamics}
When studying the dynamics of these maps, the orbits of the interval endpoints  of $\mathbb I_{\alpha}$ are of utmost importance.  We define
\[ 
\begin{aligned}
l_0&=(\alpha-1)\lambda &\textrm{ and }&\; l_n=T^n_\alpha(l_0), \textrm{ for }n\geq1\,,\\
r_0&=\; \alpha \lambda      &\textrm{ and }& \; r_n=T^n_\alpha(r_0), \textrm{ for }n\geq1\,.
\end{aligned}
\]

The {\em cylinders} for the map $T_{\alpha}$ are 
\[\Delta(\varepsilon : r) = \{ \, x \,|\, \text{sgn}(x) = \varepsilon\, \text{and}\, d_{\alpha}(x) = r\,\}\,.\]  
Letting 
\begin{equation}
\label{def: delta}
 \delta_r := \dfrac{1}{\lambda ( r + \alpha)}\;,
 \end{equation}
 for  $r$ sufficiently large, we have {\em full} cylinders (each mapped surjectively by $T_{\alpha}$ onto $\mathbb I_{\alpha}\,$)  of the form
\[ \Delta(+1:r) = (\, \delta_r, \delta_{r-1}\,]\;\; \text{and}\;  \Delta(-1:r) = [\, -\delta_{r-1}, -\delta_{r}\,)   \;.\]

For future reference, we introduce notation for strips that fiber over cylinders:  let 
\[\mathcal D(\varepsilon: r) := \{\,(x,y) \;\vert\; x \in  \Delta(\varepsilon : r)\,\}\,.\]

The standard number theoretic planar map associated to continued fractions gives here 
\begin{equation}
\label{def: T(x,y)_alpha}
\mathcal{T_\alpha}(x,y) = \left( T_\alpha(x), \frac{1}{r\lambda + \varepsilon y} \right)\;;
 \end{equation}
it is easily checked that this map has invariant measure 
\[ d\mu = \dfrac{dx\, dy}{(1 + xy)^2}\;.\]

\subsubsection{Hecke groups}  
 Rosen defined his continued fractions in order to study aspects of the Hecke groups, $G_q\subset \text{PSL}(2, \mathbb R)\,$.   With fixed index $q$ as above, let 

\begin{equation}\label{matrices}
S \, =\, 
\left[ \begin{array}{cc}
	1   &  \lambda\\
 	0   &  1
\end{array}\right], \,
 T \, =\, 
 \left[ \begin{array}{cc}
        0   &  -1\\
        1   &  0
\end{array}\right] 
\textrm{ and } \,
U =\left[ \begin{array}{cc}
	\lambda   & -1 \\
 	1  &  0
\end{array}\right] \;.
\end{equation}
Then $G_q$ is generated by any two of these, as $U = ST\,$.  In fact,  $U^q = \text{Id}$\,.  
It turns out that the $T_{1/2}$ orbit of $\lambda/2$ is   exactly the orbit of $\infty$ under powers of $U$: 
 $r_j =  U^{j+1}(\infty)\,$, where as usual we use the M\"obius (or, fractional linear) action of $2 \times 2$ matrices on the reals (extended to include $\infty$, as necessary).      From~\cite{BKS} we know
$
U^n =\left[ \begin{array}{cc}
B_{n+1} & -B_n\\
B_n & -B_{n-1}
\end{array}\right]\,
$
where the  sequence $B_n$ is 
\begin{equation}
\label{eq: Bn }
B_0 =0, \quad B_1=1, \quad B_n=\lambda B_{n-1}-B_{n-2}, \quad \textrm{ for } n=2,3,\dots \, .
\end{equation}

 \subsection{Regions of changed digits; basic deletion and addition regions}
Fix $\lambda$ and choose some $\alpha \in [\,  0,  \frac1\lambda \,]\,$.   
For $\alpha$ in the range that we consider, our intention is to {\em solve} for the region of the natural extension of $T_{\alpha}$,   using the operator 
$\mathcal T_{\alpha}\,$.  Of course, for many points $(x,y) \in \Omega_{1/2}$, we have $\mathcal T_{\alpha}(x,y) = \mathcal  T_{1/2}(x,y)\,$.  It is this last fact that we exploit,   and find that it is important 
to understand $\mathcal  T_{\alpha}(x,y)$ for all points where this value differs from $\mathcal  T_{1/2}(x,y)\,$.

 \begin{Def}\label{changeDigDef}    The {\em region of changed digits} is 
 \[ C = \{ \, (x,y) \in \Omega_{1/2} \;\;\vert\;\;   x \in \mathbb I_{\alpha} \cap \mathbb I_{1/2} \; \text{and}\;     d_{\alpha}(x)  \neq   d_{1/2}(x) \;\}\,.  \]
 \end{Def} 
 
 The region $C$ is a disjoint union of rectangles;  in general,  for each digit $d$,  the subset of $C$ determined by $x$ whose digit has changed to $d_{\alpha}(x)$ consists of two connected components, one with $x$ negative, the other with positive $x$ values.    See Figure ~\ref{figChangeDigs} for a schematic representation of this in the classical $\lambda = 1$ setting.

 \bigskip
 
 We also identify a region of $\Omega_{1/2}$ that obviously cannot be part of the new natural extension, as its marginal projection lies outside of $\mathbb I_{\alpha}\,$.   
  \begin{Def}  The basic {\em deleted region} is the $\mathcal T_{1/2}$-image of the region of changed digits:
 \[ D_0 :=  \mathcal T_{1/2}(\, C\,)\;.\]
  \end{Def}

 \begin{Lem}\label{buildDeletion}  For $\alpha < 1/2\,$, the basic deleted region is  
   \[D_0 = \{ \,(x,y) \in \Omega_{1/2} \;|\;  x  > \alpha \lambda\,\}\;.\]
 \end{Lem}
 
 \begin{proof}   The projection of $C$ to the real line is a disjoint union of intervals.     The boundaries of the  cylinders for the map $T_{z}$ are determined by the values $\frac{1}{\lambda(d+ z)}\,$;  thus,  each of the connected components of this projection lies at the end of a cylinder for $T_{1/2}$.   Since for fixed $d$ the function   $z \mapsto \frac{1}{\lambda(d+ z)}$ is decreasing, 
 the connected components lie at the right end of cylinders for negative $x$,    and at the left end of cylinders of positive $x$ values.   But,   $T_{1/2}$ itself is an increasing function for negative $x$ and a decreasing function for positive $x\,$.    Thus,  this projection is exactly the $T_{1/2}$-preimage of $(\,\alpha \lambda, \lambda/2\,]\,$.     From this, it easily follows that  $\mathcal T_{1/2}(\, C\,) =  \{ \,(x,y) \in \Omega_{1/2} \;|\;  x  > \alpha \lambda\,\}\,$.   
  \end{proof} 
  
\begin{Rmk}
 For $\alpha > 1/2\,$, one finds that $D_0 = \{ \,(x,y) \in \Omega_{1/2} \;|\;  x < (\alpha -1) \lambda \,\}\;$.
  \end{Rmk}   

We are ready to define the region that  in fact is the domain for the natural extension of $T_{\alpha}\,$.  
 \begin{Def}  Let  
 \[\Omega_{\alpha}  :=  \bigg( \, \Omega_{1/2}  \,  \bigcup \, \cup_{k=1}^{\infty} \, \mathcal T^{k}_{\alpha}(C)  \,\bigg) \setminus \bigg(  \cup_{k= 0}^{\infty}\,\mathcal T^{k}_{\alpha}(\,D_0\,) \setminus \bigg( \cup_{k= 0}^{\infty}\,\mathcal T^{k}_{\alpha}(\,D_0\,) \bigcap \,  \cup_{k=1}^{\infty} \, \mathcal T^{k}_{\alpha}(C) \,\bigg)\;\bigg) \;.\]
 \end{Def}
 \noindent
Thus,   $\Omega_{\alpha}$ is created by  adding to $\Omega_{1/2}$ all images of the regions of changed digits and then deleting all images of the basic deleted region that are not contained in the added regions.

   By analogy informed by Lemma \ref{buildDeletion}, we define the {\em basic added region} to be 
  \[ A_0 :=  \mathcal T_{\alpha}(\, C \,)\;.\]

\section{Successful quilting results in equal entropy}\label{secQuiltGood}     We sketch here a main argument for our approach.    Indeed, the results of this section show that   we will have proven part (ii) of Theorem ~\ref{thmAnnounce} once we show that  for each  $\alpha \in \left[\, \alpha_0, \omega_0 \,\right]\,$ there exists some   $k$ satisfying  the hypotheses of Proposition ~\ref{quiltGivesIso}.

We mildly informally let $\mu$ denote the probability measure on $\Omega_{\alpha}$ induced by  $d\mu = \dfrac{dx\, dy}{(1 + xy)^2}\,$.   Let $\overline{\mathcal B}_{\alpha}$  denote the Borel $\sigma$-algebra of  $\Omega_{\alpha}\,$.  
\begin{Prop}\label{quiltGivesIso}   Fix $q\in \Z, q \geq 3$ and $\lambda = \lambda_{q}= 2 \cos \frac{\pi}{q}\,$, and choose some $\alpha \in [\,0, 1/\lambda\,]$.    Let $A_0$ and $D_0$ be defined as above.   Suppose that there is some natural number $k$ such that 
\[\mathcal T^{k}_{\alpha}(A_0) = \mathcal T^{k}_{\alpha}(D_0)\,.\]
  Then  $(\mathcal T_{\alpha}, \Omega_{\alpha},\overline{ \mathcal B}_{\alpha}, \mu\,)$ is isomorphic to 
$(\mathcal T_{1/2}, \Omega_{1/2}, \overline{\mathcal B}_{1/2}, \mu\,)\,$.
\end{Prop} 

\begin{proof} By \cite{BKS}, $\mathcal T_{1/2}$  is bijective (up to $\mu$-measure zero) on $\Omega_{1/2}\,$; one easily verifies that $\mathcal T_{\alpha}$ is injective onto $\cup_{j=0}^{k-1}\; \mathcal T_{\alpha}^{j}(\, D_0\,)\,$.   Therefore, we can define   $f: \Omega_{1/2} \to \Omega_{\alpha}$ by $f$ being the identity off of $\cup_{j=0}^{k-1}\; \mathcal T_{\alpha}^{j}(\, D_0\,)$ and $f := \mathcal T_{\alpha}^{j+1} \circ \mathcal T_{1/2}^{-1}\circ \mathcal T_{\alpha}^{-j} $ on $\mathcal T_{\alpha}^{j}(\, D_0\,)\,$.     Since each of $\mathcal T_{\alpha}$ and $\mathcal T_{1/2}$ preserves the measure $\mu\,$,  one easily shows that this is an isomorphism.   
\end{proof}

\begin{Prop}  With notation as above,  let $\mu_{\alpha}$ denote the marginal measure obtained by integrating $\mu$ on the fibers of $\pi: \Omega_{\alpha} \to \mathbb I_{\alpha}\,$.  Further let $\mathcal B_{\alpha}$ denote the Borel $\sigma$-algebra of $ \mathbb I_{\alpha}\,$.  Then 
$(\mathcal T_{\alpha}, \Omega_{\alpha}, \overline{\mathcal B}_{\alpha}, \mu\,)$ is  the natural extension of  $(\,T_{\alpha}, \mathbb I_{\alpha}, \mathcal B_{\alpha}, \mu_{\alpha}\,)\,$.
 \end{Prop} 

  This is perhaps best proven using   F. ~Schweiger's  formalization of the ideas of ~\cite{NIT}, see in particular Section 22 of Schweiger's textbook ~\cite{Schw}.    For similar applications,  see \cite{DKS} or \cite{KNS}.     We leave this verification to the reader. 

\begin{Cor}\label{equiEnt}  Under the hypotheses of Proposition ~\ref {quiltGivesIso},  the systems   
 $(\,T_{\alpha}, \mathbb I_{\alpha}, \mathcal B_{\alpha}, \mu_{\alpha}\,)\,$ have the same entropy.
 \end{Cor}
    
 \begin{proof}   It is known that  a system and its natural extension have the same entropy  \cite{Roh}.    Since the natural extensions here are all isomorphic, they certainly have the same entropy.
 \end{proof}

\begin{Rmk}
\begin{itemize}
\item[(1.)]  Rohlin \cite{Roh} introduced the notion of  natural extension explicitly in order to treat entropy.
\item[(2.)]  As we display in each of Sections ~\ref{classCaseSec}, \ref{SecEven} and \ref{SecOdd},  the key to  successful quilting is equality of orbits of $r_0$ and of $l_0\,$ after a fixed number of steps.    Compare this with the discussion  of \cite{NN} relating eventual equality of these orbits and behavior of the entropy  in the classical case.  
\end{itemize}
  \end{Rmk}   

\section{Classical case, $\lambda = 1\,$:   Nakada's $\alpha$-continued fractions}\label{classCaseSec}   Aiming to maximize expository  clarity,   as an example we re-establish the form of the natural extension for Nakada's $\alpha$-continued fraction \cite{N} (thus with $\lambda = 1$)    in the case of  $\sqrt{2}-1 \leq \alpha \leq 1/2\,$.   

It is easily verified that $\alpha = 1/2$ gives the classical nearest integer continued fractions (NICF),  whose natural extension was given by \cite{N},   $(\mathcal T_{1/2}, \Omega_{1/2}, \overline{\mathcal B}_{1/2}, \mu\,)\,$,  with 
\[ \Omega_{1/2} = \bigg(\, [\, -1/2, 0\,)    \times [\, 0, g^2 \,) \;\bigg) \bigcup  \bigg(\, [\, 0, 1/2\,)    \times [\, 0, g\,) \;\bigg)\;,\] 
where $g = \frac{-1+\sqrt{5}}{2}$ is the small golden number.       (Here, the non-full cylinders of $T_{1/2}$ are 
$\Delta(+1:2) = [\,\delta_2, 1/2\,)$  and $ \Delta(-1:2) = [\, -1/2, -\delta_2\,)\,$. )

Our goal is to re-establish the following result.     

  \begin{Thm}\label{ThmClassCase}   For all $\alpha \in [\sqrt{2}-1, 1/2]$, the system 
$(\mathcal T_{\alpha}, \Omega_{\alpha}, \overline{\mathcal B}_{\alpha}, \mu\,)$ is isomorphic to 
$(\mathcal T_{1/2}, \Omega_{1/2}, \overline{\mathcal B}_{1/2}, \mu\,)\,$, where
 \[ 
\begin{aligned} 
 \Omega_{\alpha} := &  \; [\,\alpha-1, l_1\, ) \times [\,0, g^2\,)\;\;  \bigcup  \\
                                 & \; [\,l_1, r_1 \, ) \times \bigg(\,  [\,0, g^2\,) \cup [\, 1/2, g\,)\, \bigg)\;\;  \bigcup  \\
                                 &\;  [\, r_1, \alpha \, ) \times   [\,0,   g\,)\;,
\end{aligned} 
\]  
and  $\mathcal B_{\alpha}$  denotes the $\sigma$-algebra of $\mu$-Borel subsets of $\Omega_{\alpha}\,$.

Furthermore,   the entropy of $T_{\alpha}$ with respect to the marginal measure of the above system equals  $\dfrac{1}{\ln (1+g)}\, \dfrac{\pi^2}{6}\,$.
\end{Thm}

\begin{Rmk}  The constancy of  the entropy in this setting was first established by Moussa, Cassa and Marmi \cite{MCM}, the value at $\alpha = 1/2$ having been determined by \cite{N}.     We also note that Nakada and Natsui \cite{NN} explicitly show the isomorphism of these natural extensions (see their Appendix).
\end{Rmk}

  \subsection {Explicit form of the basic addition region}  
Fix $(\mathcal S, \Omega) :=(\,  \mathcal T_{1/2},  \Omega_{1/2}\,)$ and consider also a fixed $\alpha \in [\sqrt{2}-1, 1/2\,)\,$.  
   Recall that our intention is to  solve  for the region of the natural extension of $T_{\alpha}$,    by finding the $\mathcal T := \mathcal T_{\alpha}$-orbits of basic added and deleted regions.  
 \begin{Rmk}    Other than $\mathcal S$ and $\Omega$, all notation refers to values dependent upon $\alpha$ unless explicit dependence upon $1/2$ is indicated.  
 \end{Rmk}

 Note that with $r$ fixed,  $\delta_r$ as defined in Equation ~\eqref{def: delta} is a decreasing function in $\alpha\,$.    Note also that for  $\alpha \in (\sqrt{2}-1, 1/2)$ one has $\delta_2(\alpha) < \alpha\,$.  
 
\begin{figure}
\begin{center}
\scalebox{.9}{\includegraphics{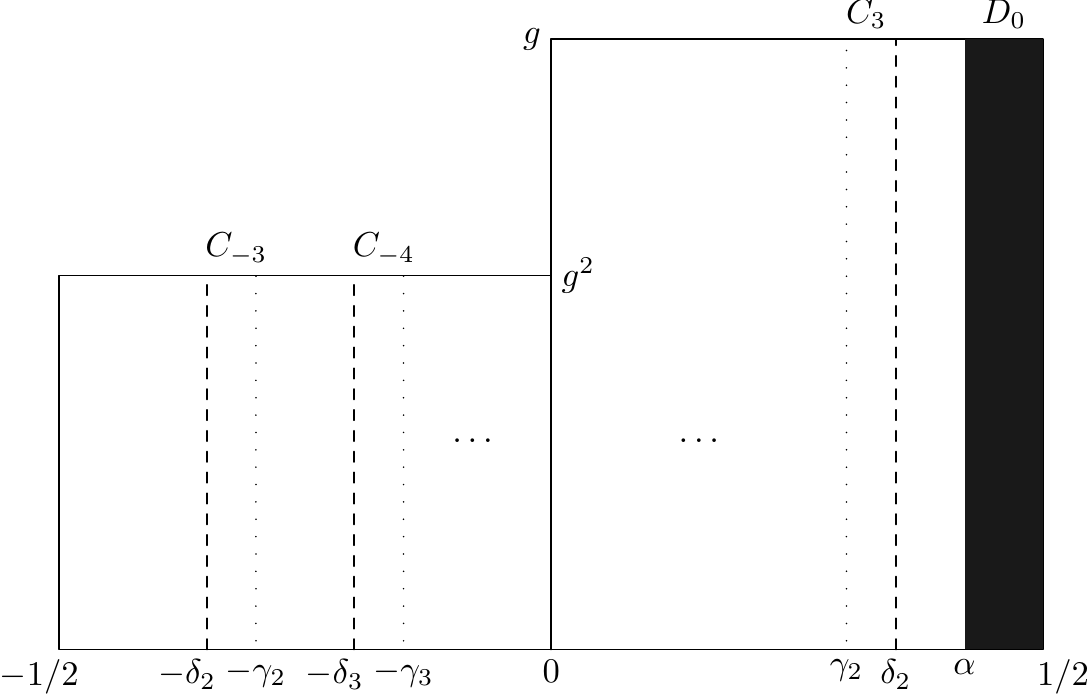}}
\caption{Regions of change of digit and basic deletion.  Here each $\gamma_{m}$ represents $\delta_m(1/2)\,$.   To aid visualization, we use here $C_{\varepsilon r} = C \cap \Delta( \varepsilon: r)\,$.}
\label{figChangeDigs}
\end{center}
\end{figure}

\begin{Lem}\label{basicAddition}  The basic added region is  given by
 \[ A_0 =  [\, \alpha-1, -1/2) \times  [0, g^2 )\;.\]  
 \end{Lem}

\begin{proof}    Similarly to the proof of Lemma ~\ref{buildDeletion},  we note that for $x<0$ each connected component of the marginal projection of $C$ lies at the left end of  a $T_{\alpha}$-cylinder and is sent by (the locally increasing function) $T_{\alpha}$ to $ [\, \alpha-1, -1/2)$.    When $x>0\,$,  the component lies at the right end of its cylinder and is also sent by $T_{\alpha}$ to $ [\, \alpha-1, -1/2)$.   One trivially checks  that no other points of    $\mathbb I_{\alpha}$ are sent to this subinterval.  
 
We now discuss the $y$-coordinates of points in $\mathcal T(A_0\,)\,$.     The fibers of $\Omega_{1/2}$ above  cylinders where $x<0$ are of the form $[0, g^2)\,$, while above cylinders with $x>0$ the fibers are of the form $[0,g)\,$.    Now,  for $(x,y)  \in \mathcal D(\varepsilon: r)\,$, 
the $y$-coordinate of  $\mathcal T (x,y)$ is $1/(r + \varepsilon y)\,$.    Thus,      $\mathcal T$ sends   $\mathcal D(-1: r)\,$ and $\mathcal D(+1: r)$ to horizontal strips whose $y$-values lie in $[\, \frac{1}{r  + g},\, \frac{1}{r  -g^2}\,]\,$.   But,   
 $g^2 = 1-g\,$ and thus this image lies directly above that of the horizontal strip given by $\mathcal T$ applied to  $\mathcal D(-1: r+1)\,$ and $\mathcal D(+1: r+1)\,$.     
 
 Now, the greatest $y$-value of $\mathcal T(C)$  comes from the intersection of the projection of $C$ with  $\mathcal D(-1:3)\,$.     Thus,   $A_0 =  [\, \alpha-1, -1/2) \times  [0, 1/(3-g^2)\, )\,$.   
 Since $1/(3-g^2) = g^2$, the result follows. 
 \end{proof}

\subsection{Quilting}
  We now show that the $\mathcal T$-orbits of the added regions eventually match the orbits of the deleted regions, and thus $\mathcal T$ is bijective (modulo $\mu$-measure zero) on $\Omega_{\alpha}\,$. 

\begin{Lem}\label{orbitsMatch}  The following equality holds:
 \[ l_2 = r_2 \;.\]  
 Furthermore,  there is a $d \in \mathbb N$ such that 
 \[ l_1 \in \Delta(-1:d)\;\; \text{and}\;\; r_1 \in \Delta(1:d-1)\;.\]
 \end{Lem}

\begin{proof}    We have that $l_0 \in \Delta(-1:2)$ and $r_0 \in  \Delta(+1:2)\,$, giving 
\[ l_1 = \dfrac{2 \alpha - 1}{1-\alpha}\;\; \text{and}\;\; r_1 =  \dfrac{1- 2 \alpha}{\alpha}\;.\]
Therefore, 
\[ l_2 = \dfrac{1-\alpha}{1- 2 \alpha} - d\;\; \text{and}\;\; r_2 =  \dfrac{\alpha}{1- 2 \alpha}- d'\;,\]
with $d, d'$ the appropriate $T_{\alpha}$-digits.  Now, $l_2 - r_2 = 1 + d' - d$ and is the difference of two elements in $\mathbb I_{\alpha}\,$, an interval of length one.   We thus conclude both that $d' = d-1$ and $l_2 = r_2\,$.
 \end{proof}

The orbit of the basic addition region, $A_0\,$,  is quickly synchronized with that of the basic deletion region, $D_0\,$.    Recall that $\mathcal D(\varepsilon:r)$ fibers over $\Delta(\varepsilon:r)\,$.
\begin{Lem}\label{filling in holes}  We have 
 \[ \mathcal T^2(\,A_0\,)  =  \mathcal T^2(\,D_0\,) \;.\]  
 \end{Lem}

\begin{proof}   Let $A_1 :=  \mathcal T(\,A_0\,)\,$; since $ A_0 \subset   \mathcal D(-1:2)$, an elementary calculation shows that this is 
\[A_1 = [\, l_1\,, 0) \times [\,1/2, g\,)\;.\]
Similarly,   defining $D_1 :=  \mathcal T(\,D_0\,)\,$,  one has $D_1 \subset  \mathcal D(1:2)$,
 and finds
\[D_1 = [\, 0, r_1\, ) \times [\,g^2, 1/2\,)\;.\]
 
With $d$ as in Lemma ~\ref{orbitsMatch},  let 
\[A'_1 := A_1 \cap \mathcal D(-1:d)\;\; \text{ and} \;\; D'_1 := D_1 \cap \mathcal D(+1:d-1)\;.\]
By elementary calculation, Lemma ~\ref{orbitsMatch}, and an application of the identity $g^2 = 1-g\,$,  one finds that 
\[ 
\begin{aligned}
\mathcal T(A'_1) &= \, \mathcal T(D'_1) \\
                              &=  \left[\, l_2,  \alpha \right) \times \left[\, \frac{2}{2d - 1}, d-g\,\right)\;.
\end{aligned}
\]
Each of  $A_1  \setminus A'_1$ and $D_1  \setminus D'_1$  projects to the union of full cylinders:
\[ A_1  \setminus A'_1 = \bigcup_{m= d+1}^{\infty}\, \Delta(-1: -m) \times [\, 1/2, g\,)\;\;\text{and} \;\; D_1  \setminus D'_1 = \bigcup_{m= d}^{\infty}\, \Delta(+1: m) \times [\, 1/2, g^2\,)\;.\]
Since $\mathcal T\big(\,  \Delta(-1: -m) \times [\, 1/2, g\,)\,\big) = \mathcal T\big(\, \Delta(+1: m-1) \times [\, 1/2, g^2\,)\,\big)\,,$ we conclude that  $\mathcal T(\,A_1\,)  =  \mathcal T(\,D_1\,) \,$, and the result follows.
 \end{proof}

\subsection{Isomorphic systems} 
In this subsection, we complete the proof of Theorem ~\ref{ThmClassCase}.
\begin{figure}
\begin{center}
\scalebox{0.8}{\includegraphics{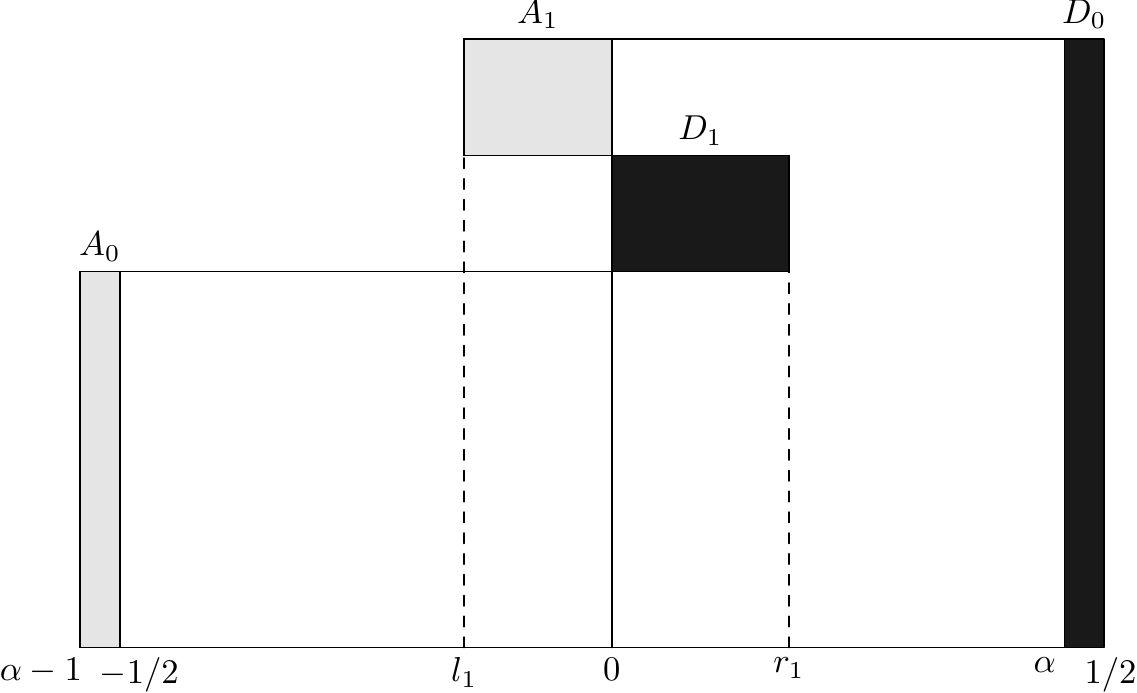}}
\caption{ Representative region $\Omega_{\alpha}$ for $\alpha \in (\, \sqrt{2}-1,\, 1/2\,)\,$; $q=3\,$.}
\label{figOmegaAlpha}
\end{center}
\end{figure}

\begin{Cor}\label{exactRegion}  We have 
 \[ 
\begin{aligned} 
 \Omega_{\alpha} = &  \; [\,\alpha-1, l_1\, ) \times [\,0, g^2\,)\;\;  \bigcup  \\
                                 & \; [\,l_1, r_1 \, ) \times \bigg(\,  [\,0, g^2\,) \cup [\, 1/2, g\,)\, \bigg)\;\;  \bigcup  \\
                                 &\;  [\, r_1, \alpha \, ) \times   [\,0,   g\,)\;.
\end{aligned} 
\]  
  Furthermore,   the $\mu$-area of $\Omega_{\alpha}$ equals that of $\Omega_{1/2}\,$.
 \end{Cor}
\begin{proof}   From the above,   
\[\Omega_{\alpha} =  \bigg(\,  \Omega \cup A_0 \cup A_1\, \bigg) \setminus \bigg(\, D_0 \cup D_1\,\bigg)\;.\]
The explicit shape of $\Omega_{\alpha}$ follows.   

By definition of $A_0$  and by Lemma ~\ref{buildDeletion},  $A_0$ and $D_0$ are the images of the union of the change of digit regions under $\mathcal T$ and $\mathcal S$, respectively.   Since both $\mathcal S$ and $\mathcal T$ are $\mu$-measure preserving, thus $\mu(D_0) = \mu(A_0)\,$.   Again since $\mathcal T$ preserves measure, we also have  $\mu(D_1) = \mu(A_1)\,$. 
 \end{proof}

 Of course, the equality of the areas is already implied by the arguments of Section ~\ref{secQuiltGood}.  Indeed, since the entropy of the NICF is known, those arguments can be easily adapted to complete the proof of the Theorem ~\ref{ThmClassCase}.

\section{Even $q\,$;  $\alpha \in (\, \alpha_0, 1/2 \,]\,$}\label{SecEven}  
The natural extension for the Rosen fractions was determined in \cite{BKS}.      The exact form of the domain depends on the parity of the index $q$.   
\subsection{Natural extensions for Rosen fractions}  Let $q=2p$ for $p \in \N$ and $p\geq 2\,$.     The domain, see Figure ~\ref{im: natural extension 1/2},  given by \cite{BKS} is 
 \[ \Omega = \bigcup_{j=1}^p J_j \times K_j\,, \]
here $J_j$ is defined as follows: Let $\varphi_j=T_{1/2}^j\left(
-\frac{\lambda}{2} \right)$, then $J_j=[\varphi_{j-1},\varphi_{j})$ for
$j \in \{1,2,\ldots,p-1 \}$ and $J_p = \left[ 0,\frac{\lambda}{2}
\right)$. Further, $K_j=[0,L_j]$ for $j \in \{1,2,\ldots,p-1 \}$
and $K_p = [0,1]\,$, where the values of the $L_j$   are given by  the
following relations:
\[
\begin{aligned}
L_1\, &=\, 1/(\lambda + 1),\\
L_j\, &=\, 1/(\lambda -L_{j-1})\qquad \text{for }\,j\in \{ 2,\dots ,p-1 \}.
\end{aligned}
\]

The normalizing constant $C$ such that $C d\mu$ gives a probability measure on $\Omega_{\alpha}$ is 
\begin{equation}\label{evenNormConst}
C = \dfrac{1}{\ln[\,(1 + \cos \pi/q\,)/\sin \pi/q\,]}\;,
\end{equation}
see Lemma 3.2 of \cite{BKS}.

\begin{figure}[!ht]
\includegraphics[height=40mm]{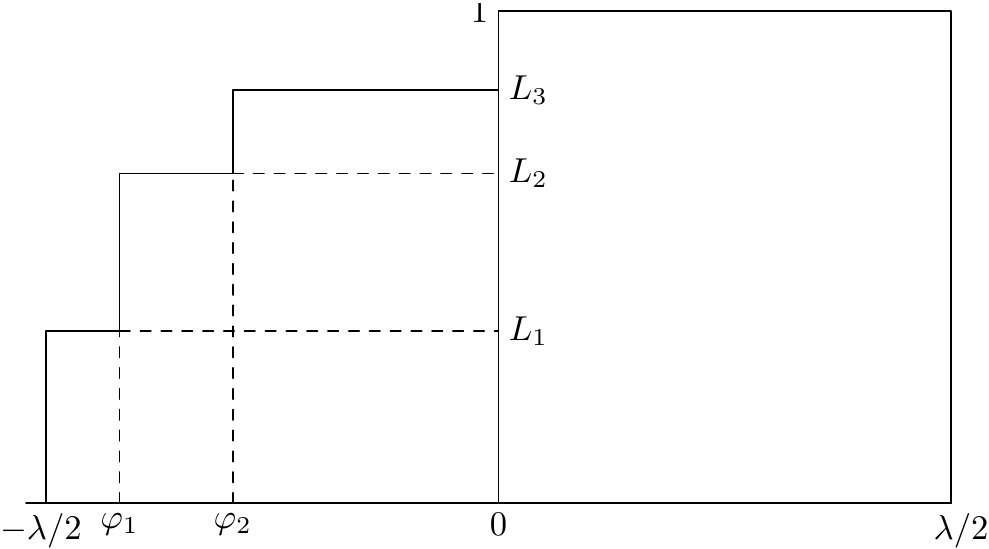}
\caption{The natural extension for $q=8$ and $\alpha = \frac12$.}
\label{im: natural extension 1/2}
\end{figure}

In this case, using that $U$ is of order $q$ (in the projective group) for the $B_j$ as defined in \eqref{eq: Bn },  we have that
\begin{equation}\label{bRelation}
B_{p-1} = B_{p+1} = \frac{\lambda}{2} B_p \; \textrm{ and }\;  B_{p-2} = \left( \frac{\lambda^2}{2}-1 \right) B_p\,.
\end{equation}

The key to understanding the system for $\mathcal T_{\alpha}$ is that  for all $q= 2p\,$, one has that $l_p = r_p\,$ for  all  of our $\alpha\,$.    In proving this,  it is convenient to use the fact that the orbits of $r_0$ and $-r_0$ coincide after one application of $T_{\alpha}\,$.

\begin{Lem}
\label{lem:rZeroDigsEven}
For any $\alpha > \alpha_0$, the $T_{\alpha}$-expansion of both $l_0$ and $-r_0$ starts as \newline $[(-1:1)^{p-1}, \dots\,]\,$.
\end{Lem}
\begin{proof}
A point that lies to the left of the appropriate $p-2{\text{nd}}$ pre-image of $-\delta_1$ has ones for its first $p-1$ digits. This pre-image is given by $(S^{-1}T)^{-p+2}(-\delta_1)\,$, where we are using the matrices defined in Equation \eqref{matrices}.
 
But, 
\[
\begin{aligned}
(S^{-1}T)^{-p+2} &= (S^{-1}U^{-1}S)^{-p+2}\\
\\
                              & =  S^{-1}U^{p-2}S\\
                              \\
                             &= 
\left[ \begin{array}{cc}
	B_{p-1} -\lambda B_{p-2}   &  -B_{p-2}\\
	\\
 	B_{p-2}   &  B_{p-1}
\end{array}\right]\;.
\end{aligned}
\]
\noindent
Now, the relations of Equation ~\eqref{bRelation} give
\[ (S^{-1}T)^{-p+2} =   \frac{B_p}{2} 
\left[ \begin{array}{cc}
	 -\lambda^3 +3\lambda   &  -\lambda^2 + 2 \\
 	\lambda^2 -2    & \lambda
\end{array}\right]\;.
\]
Thus the $p-2$nd pre-image of $ - \delta_1$ is given by
$$
T_{\alpha}^{-p+2}(-\delta_1)= \frac{(\lambda^3-3\lambda)\delta_1-\lambda^2+2}{(-\lambda^2+2)\delta_1 +\lambda}=  - \frac{\alpha(\lambda^3-2\lambda)+\lambda}{\alpha\lambda^2+2}.
$$

In particular,  $-r_0$ starts with $p-1$ digits one if 
$$-r_0 = -\alpha \lambda <  - \frac{\alpha(\lambda^3-2\lambda)+\lambda}{\alpha\lambda^2+2}.$$ 
Rewriting this inequality yields that it holds whenever 
$$\alpha>\alpha_0 = \frac{\lambda^2-4+\sqrt{(4-\lambda^2)^2+4\lambda^2}}{2\lambda^2}.
$$
Finally, if $\alpha > \alpha_0$, then it immediately follows that $l_0$ also starts with $p-1$ ones since  $l_0 = (\alpha-1)\lambda <  -\alpha \lambda = -r_0$.
\end{proof}

\begin{Lem}\label{lemEvenOrbitMatch}  For  $\alpha \in  (\alpha_0, 1/2)\,$, $r_p = l_p\,$.    Furthermore,  there is $d \in \mathbb N$ such that 
 \[ l_{p-1} \in \Delta(-1:d)\;\; \text{and}\;\; r_{p-1} \in \Delta(1:d-1)\;.\]
\end{Lem}

\begin{proof} 
Using the previous lemma, we find
\[
\begin{aligned}
l_{p-1} = T_{\alpha}^{p-1} (l_0) &=  (S^{-1}T)^{p-1}\big(\, (\alpha-1)\lambda\, \big)\\
                                                        &=  \left[ \begin{array}{cc}
	 -\lambda    &  -1 \\
 	1  &0
\end{array}\right]  \left[ \begin{array}{cc}
	  \lambda&  \lambda^2 - 2 \\
 	-\lambda^2 +2    & -\lambda^3 +3\lambda   
\end{array}\right] (\, (\alpha-1)\lambda )\\
&=  \dfrac{(1-2 \alpha)}{\alpha \lambda^2 - 2}\;,
\end{aligned}
\]
and similarly, since  $T_{\alpha}(r_0) = T_{\alpha}(-r_0)$,  
\[ 
r_{p-1} = T_{\alpha}^{p-1} (-r_0)  = \dfrac{(2 \alpha - 1) \lambda}{(1-\alpha) \lambda^2 - 2}\;.
\]

Therefore,  
\[
\left|\, \dfrac{1}{l_{p-1}}\,\right| - \left|\, \dfrac{1}{r_{p-1}}\,\right| = \lambda\;.
\]
But,  then $T_{\alpha}(l_{p-1}) - T_{\alpha}(r_{p-1})$ is an integer multiple of $\lambda$.   However,  this is the difference of two elements of $\mathbb I_{\alpha}$, and thus this multiple must be zero.  We conclude that  $r_p = l_p\,$.   That the digits of these points is as claimed follows as in the classical case.
\end{proof}

\begin{figure}[!ht]
\includegraphics[height=45mm]{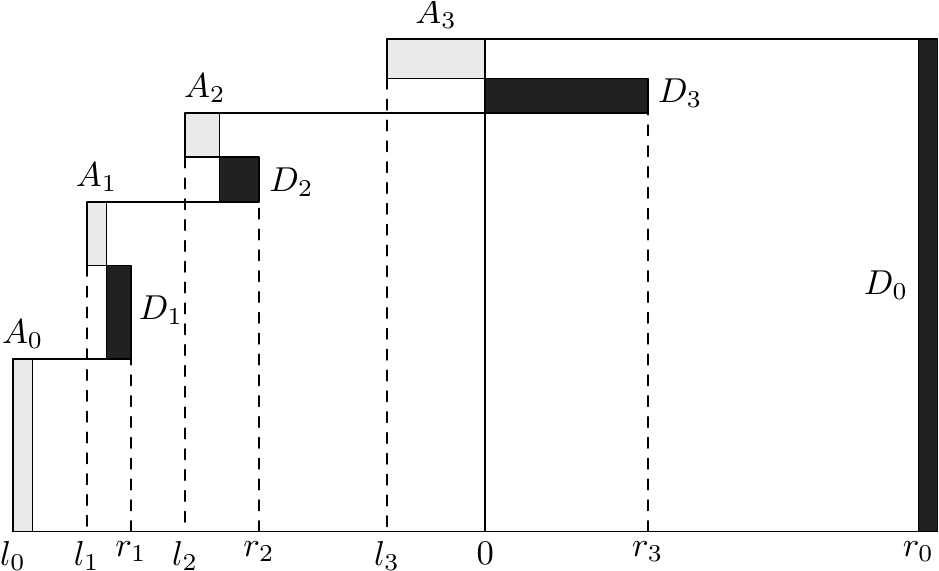}
\caption{The added and deleted rectangles for  $\alpha = 0.48$ in the natural extensions for $q=8\,$.   Here $A_i, D_i$ denote $\mathcal T^i(A_0), \mathcal T^i(D_0) \,$, respectively.}
\label{im: natural extension q=8 a =0.48}
\end{figure}

To describe the initial orbits of $A_0$ and $D_0$, we use the following sequence.  
\[ H_1 = \frac{1}{\lambda} \;\; \text{and} \;\; H_i = \frac{1}{\lambda-H_{i-1}}\;\; \text{ for}\; i \geq 2\,.\]

\begin{Lem}  For  $\alpha \in  (\alpha_0, 1/2)\,$, $\mathcal T^p(A_0) = \mathcal T^p(D_0)\,$.
\end{Lem}
\begin{proof}  One easily checks that the basic added region  and basic deleted regions are   
\[  A_0 = [\,l_0, - \lambda/2\,) \times [0, L_1)\,\;\text{and} \; \;D_0 = (r_0, \lambda/2\, ] \times [\,0, R)\,.\]
Since $ - \lambda/2 < - r_0\,$,  from Lemma ~\ref{lem:rZeroDigsEven}   we find that  all $x\in [\,l_0, - \lambda/2\,)$ share the same first $p-1$ of their $T_{\alpha}$-digits, and in fact that $T_{\alpha}^{p-1}(-\lambda/2)  = T_{1/2}^{p-1}(-\lambda/2)\,$.   But,  $T_{1/2}^{p-1}(-\lambda/2) = 0\,$.    Thus, recalling that $1 = 1/(\lambda - L_{p-1})$, we find that 
\[ \mathcal T^{p-1}(A_0) =  [\,l_{p-1}, 0) \times [\, H_{p-1}, 1\,)\;.\]
 
Paying attention to sign and orientation,  one finds that  $\mathcal T(A_0) =  [\, \varphi_1, r_1) \times [\, L_1, H_{1}\,)\,$, and thus
\[ \mathcal T^{p-1}(D_0) =  [\,0, r_{p-1}) \times [\, L_{p-1}, H_{p-1}\,)\;.\]

Analogously to the classical case,   with $d$ from Lemma~\ref{lemEvenOrbitMatch}, we let 
\[ A'_{p-1} := \mathcal T^{p-1}(A_0) \cap \mathcal D(-1: d)\,,\]
and 
\[ D'_{p-1} := \mathcal T^{p-1}(D_0) \cap \mathcal D(+1: d-1)\,.\]
Arguing analogously to the classical case, we find  that the $\mathcal T$  images of $A'_{p-1}$ and $D'_{p-1}$ agree, and then that the various layers from each of $ \mathcal T^{p-1}(A_0) \setminus A'_{p-1}$ and of $ \mathcal T^{p-1}(D_0) \setminus D'_{p-1}$ also agree.    The result follows. 
\end{proof}

\begin{figure}[!ht]
\scalebox{0.9}{\includegraphics[height=40mm]{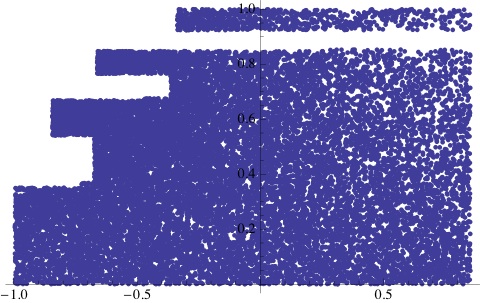}}
\scalebox{0.9}{\includegraphics[height=40mm]{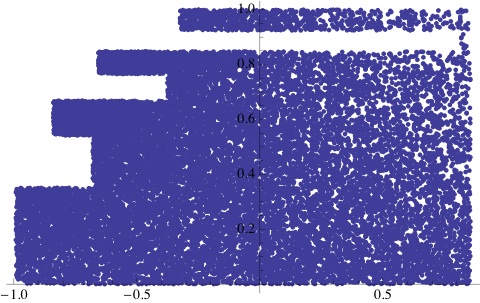}}
\caption{Change of topology at $\alpha = \alpha_0\,$:  Simulations of the natural extension for $q=8$ with on the left $\alpha = \alpha_1 - 0.001$ and on the right $\alpha = \alpha_1 +0.001$.}
\label{im: natural extension alpha bigger bound}
\end{figure}

\begin{Lem}
\label{lem:smallerAlpNotConnEven}
The region   $\Omega_{\alpha_0}$ is not connected.
\end{Lem}
\begin{proof}    The proof of Lemma ~\ref{lem:rZeroDigsEven}  shows that if $\alpha = \alpha_0\,$, then $r_{p-2} = - \delta_1\,$.   But then 
\[ \mathcal T^{p-1}(D_0) =  [\,0, \alpha \lambda) \times [\, L_{p-1}, H_{p-1}\,)\;.\]
Thus,  since  the $\mathcal T$-orbit  of $A_0$  can never fill in this deleted strip,  we see that $\Omega_{\alpha_0}$ is indeed disconnected.    
\end{proof}

The arguments of Section ~\ref{secQuiltGood} then finish the proof of Theorem ~\ref{thmAnnounce} in this case of even index $q\,$.

\section{Odd $q\,$;  $\alpha \in (\, \alpha_0,1/2 \,]\,$}\label{SecOdd}  
Let $q=2h+3\,$, for $h\geq 1$.  
\subsection{Natural extensions for Rosen fractions} 
We recycle notation, now using $\varphi_j$ and $L_j$ as follows (all
necessary calculations are in~\cite{BKS}):
$$
\varphi_{0} = - \frac{\lambda}{2},\qquad \text{and}\quad \varphi_{j} =
T_{1/2}^{j} \left( - \frac{\lambda}{2} \right),\,\, 0 \le j \le 2h
+1.
$$
We recall that
\[
 \begin{array}{l} -\frac{\lambda}{2} \le \varphi_{j} < -
\frac{2}{3 \lambda} \quad \text{for} \quad j \in \{0,1,\ldots
,h -1
\}\, \cup \, \{h +1,\ldots ,2h \}\,, \\
\\
-\frac{2}{3 \lambda} < \varphi_{h} < - \frac{2}{5\lambda}\,,\; \text{and} \\
\\
\varphi_{2h+1} = 0\,.
\end{array}
\]
Also we put, with $R$ the positive root of $R^2+(2-\lambda )R-1=0$,
\[
\begin{array}{ccl}
L_1 &  =  &  \dfrac{1}{2 \lambda - L_{2h}}\,,\\
\\
L_2 &  =  &  \dfrac{1}{2 \lambda - L_{2h +1}}\,, \\
\\
L_{j} & = & \dfrac{1}{\lambda - L_{j-2}}, \quad 2 < j \le 2h +2
\; .
\end{array} 
\]

The domain, see Figure~\ref{im:oddNatExt1/2} ,  given in \cite{BKS} is  $\Omega\, =\, \bigcup_{j=1}^{2h+2} J_j\times K_j\,$,
where 
\[
\begin{aligned}
J_{2k}& = [\varphi_{h+k},\varphi_k), {\mbox{ for }} k\in \{ 1,\cdots , h \} ,\\
J_{2k+1}& = [\varphi_k,\varphi_{h+k+1}), {\mbox{ for }} k\in \{ 0,1,\cdots , h \} ,
\end{aligned}
\]
and $J_{2h+2}\, =\, [0,\frac{\lambda}{2})$. Let $K_j=[0,L_j]$ for
$j\in \{ 1,\cdots , 2h+1 \} $ and let $K_{2h+2} =[0,R]\,$.

Here, the normalizing constant $C$ such that $C d\mu$ gives a probability measure on $\Omega_{\alpha}$ is 
\begin{equation}\label{oddNormConst}
C = \dfrac{1}{\ln(1 + R\,)}\;,
\end{equation}
see Lemma 3.4 of \cite{BKS}.

\begin{figure}[!ht]
\includegraphics[height=40mm]{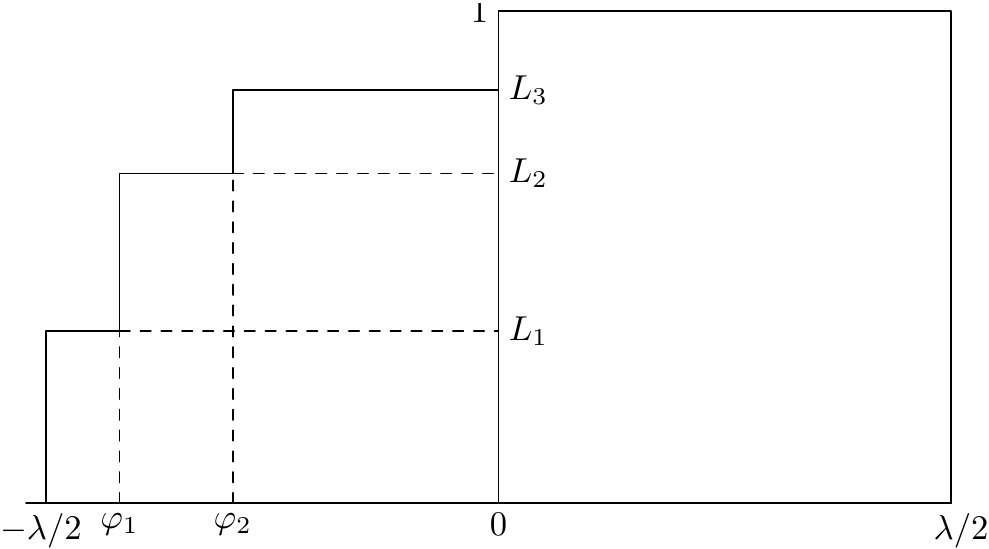}
\caption{The natural extension for $q=7$ and $\alpha = \frac12$.}
\label{im:oddNatExt1/2}
\end{figure}

  Similarly to the even case, for each of our $\alpha$,  the basic added region  and basic deleted regions are given by 
\[  A_0 = [\,l_0, - \lambda/2) \times [0, L_1)\,,\;\; D_0 = (r_0, \lambda/2] \times [0, R)\,.\]
In this case, using that $U$ is of order $2h+3$ we have that
\begin{equation}
\label{eq: B relations odd}
B_{h+1} = B_{h+2}, \quad B_h = (\lambda-1)B_{h+1} \quad \textrm{and} \quad  B_{h-1} = (\lambda^2-\lambda-1)B_{h+1}. 
\end{equation}

In this section we prove that for $q=2h+3$ one has that $l_{2h+2}=r_{2h+2}$ for all $\alpha \in \left(\alpha_0, \frac12\right)$ and consequently that the added blocks coincide with the deleted blocks after $2h+2$ steps.  Here also,  we use the fact that the orbits of $r_0$ and $-r_0$ coincide after one application of $T_{\alpha}\,$.

\begin{Lem}
\label{lem: alpha_0 odd}
For any $\alpha > \alpha_0$, the $T_{\alpha}$-expansion of both $l_0$ and $-r_0$ starts as \newline $[\,(-1:1)^h,(-1,2),(-1,1)^h,\dots\,]\,$.
\end{Lem}
\begin{proof}
We recall that $-\delta_1 < \varphi_h < -\delta_2$.  We certainly have that $l_0 < \varphi_0 <r_0\,$;  it immediately follows that $T_{\alpha}^{h}(l_0) < \varphi_h< -\delta_2$ and that $T_{\alpha}^{h}(-r_0) > \varphi_h> -\delta_1$. To prove the lemma it suffices to show that  $T_{\alpha}^{h}(l_0) > -\delta_1$, $T_{\alpha}^{h}(r_0) < -\delta_2$ and  $T_{\alpha}^{2h}(-r_0) < \delta_1$.    We only show the last of these inequalities, because one can easily check that it imposes the strongest restriction on the value of $\alpha\,$.   

Assuming the first two conditions are met we find that
\[
T_{\alpha}^{2h}(-r_0)= (S^{-1}T)^{h-1}S^{-2}T(S^{-1}T)^h (-r_0).
\]
Again,
\[
\begin{aligned}
(S^{-1}T)^h &= (S^{-1}U^{-1}S)^h \\
\\
                              & =  S^{-1}U^{-h}S                              \\
\\
                             &= 
\left[ \begin{array}{cc}
	-B_{h+2}    &  -B_{h+3}\\
 	B_{h+3}   &  B_{h+4}
\end{array}\right]\;.
\end{aligned}
\]
It easily follows that
$$
(S^{-1}T)^{h-1} = \left[ \begin{array}{cc}
	-B_{h+3}    &  -B_{h+4}\\
 	B_{h+4}   &  B_{h+5}
\end{array}\right]\;.
$$

We find that
\[
\begin{aligned}
(S^{-1}T)^{h-1}S^{-2}T(S^{-1}T)^h&= \left[ \begin{array}{cc}
	-B_{h+3}    &  -B_{h+4}\\
 	B_{h+4}   &  B_{h+5}
\end{array}\right]\;
\left[ \begin{array}{cc}
	-2 \lambda     &  -1 \\
 	1 &  0
\end{array}\right]\;
\left[ \begin{array}{cc}
	-B_{h+2}    &  -B_{h+3}\\
 	B_{h+3}   &  B_{h+4}
\end{array}\right] \\
\\
&= 
\left[ \begin{array}{cc}
	-B^2_{h+2}-B_{h+1}B_{h+3}    &  -2B_{h+2}B_{h+3}\\
	\\
 	B_{h+2}B_{h+3} +B_{h+1}B_{h+4}  &  B_{h+2}B_{h+4}+B^2_{h+3}
\end{array}\right] .
\end{aligned}
\]
Using the relations~(\ref{eq: B relations odd}) yields that
\[
T_{\alpha}^{2h}(-r_0)= B^2_{h+1}
\left[ \begin{array}{cc}
	\lambda     &  2(\lambda -1)\\
 	2 -\lambda^2   &   (3-2\lambda )\lambda
\end{array}\right].
\]
Substituting $r_0 = \alpha \lambda $ gives
$$
T_{\alpha}^{2h}(-r_0) =\frac{-\alpha \lambda^2 +2(\lambda -1)}{\alpha \lambda^3 -2\alpha\lambda +(3-2\lambda)\lambda }.
$$
A calculation shows that $\frac{-\alpha \lambda^2 +2(\lambda -1)}{\alpha \lambda^3 -2\alpha\lambda +(3-2\lambda)\lambda } < -\delta_1$ whenever
$$
\alpha > \alpha_0 = \frac{\lambda -2 +\sqrt{2\lambda^2-4\lambda+4}}{\lambda^2}\,.
$$
\end{proof}

\begin{Lem}
For $\alpha \in (\alpha_0,1/2)$ one has that $\mathcal{T}^{2h+2} (A_0) = \mathcal{T}^{2h+2} (D_0)\,$. 
\end{Lem}
\begin{proof}
Thanks to Lemma~\ref{lem: alpha_0 odd} it is easy to compute $l_{2h+1}$ and $r_{2h+1}\,$.   We have
\[
\begin{aligned}
l_{2h+1} = T^{2h+1} (l_0) &= \left[ \begin{array}{cc}
	-2   &  -\lambda \\
 	\lambda    &  2 \lambda -2
\end{array}\right]\left((\alpha-1)\lambda \right) \\
\\
& = \frac{-2(\alpha-1)\lambda - \lambda }{(\alpha-1)\lambda^2+2\lambda -2 }\,,\\
\\
r_{2h+1} = T^{2h+1} (r_0) &= \left[ \begin{array}{cc}
	-2   &  -\lambda \\
 	\lambda    &  2 \lambda -2
\end{array}\right]\left(-\alpha \lambda \right) \\
\\
& = \frac{-2\alpha\lambda - \lambda }{\alpha\lambda^2+2\lambda -2 }\;.\\
\end{aligned}
\]
We thus find that
$$
\left|\frac{1}{r_{2h+1}}\right| - \left|\frac{1}{l_{2h+1}}\right| = \frac{-2\alpha\lambda^2+\lambda^2  }{(2\alpha-1)\lambda } = -\lambda\,.
$$
Arguments completely analogous to the even case now give that  that $l_{2h+2} = r_{2h+2}$ and in fact that $\mathcal{T}^{2h+2} (A_0) = \mathcal{T}^{2h+2} (D_0)\,$.
\end{proof}

\begin{figure}[!ht]
\includegraphics[height=40mm]{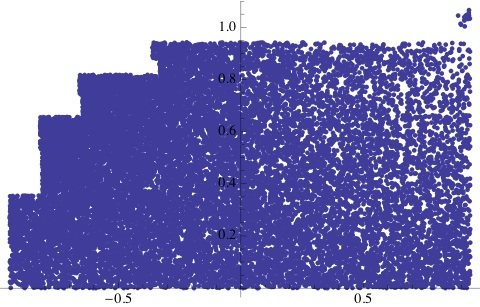}
\includegraphics[height=40mm]{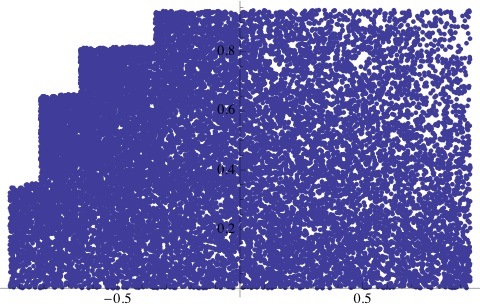}
\caption{Change of topology at $\alpha = \alpha_0\,$: Simulations of the natural extension for $q=9\,$;  on  left   $\alpha = \alpha_1 - 0.001\,$, on right $\alpha = \alpha_0 +0.001$.}
\label{im: natural extension alpha bigger bound}
\end{figure}

Arguing as in Lemma ~\ref{lem:smallerAlpNotConnEven} and completing with the arguments of Section ~\ref{secQuiltGood}  finish the proof of Theorem ~\ref{thmAnnounce} in this case of odd index $q\,$.

\section{Large $\alpha$,  by way of Dajani {\em et al}}\label{secDKSlargeAlpha}   
We finish the proof of Theorem ~\ref{thmAnnounce} by appropriately interpreting results of \cite{DKS}.

\subsection{Successful quilting for $\alpha \in (1/2,  \omega_0\,]$}\label{subsecDKSquilts}    
As already stated in our introduction, \cite{DKS} shows that the domain $\Omega_{\alpha}$ is connected for all $q$ and all $\alpha \in (\,1/2, 1/\lambda\,]\,$.     Thus with the above,  part (i) of Theorem ~\ref{thmAnnounce} follows.   

For even $q$ and $\alpha \in (\,1/2, 1/\lambda\,]\,$,  Theorem 2.2 of \cite{DKS}  shows that $r_p = l_p$ and that the digits of $r_{p-1}$ and $l_{p-1}$ agree up to sign and a shift (of the correct sign) by one.   From this, just as above,  one can in fact show that $\mathcal T^p(A_0)  = \mathcal T^p(D_0)$ in these cases as well.   

For odd $q$ and $\alpha \in (\,1/2, \omega_0\,)\,$,  Theorem 2.9 of \cite{DKS}  shows that $r_{2h +2} = l_{2h + 2}$ and that the digits of their orbit predecessors agree up to sign and a shift by one.   Thus, here  one can show that $\mathcal T^{2h +2}(A_0)  = \mathcal T^{2h +2}(D_0)\,$.

Using the results of Section ~\ref{secQuiltGood}, we thus have that part (ii) of Theorem ~\ref{thmAnnounce}  holds.

\subsection{Nearly successful quilting and unequal entropy}\label{subsecQuiltBad}   

 We now show that the entropy of $T_{\alpha}$  is {\em not} equal to that of $T_{1/2}$ for $\alpha_0 > \omega_0$ in the case of odd $q\,$.     Indeed, for these values,   the results of \cite{DKS} show that although the natural extensions remain connected,   the conditions for successful quilting are not fully satisfied.  
 
\begin{Lem}\label{quiltNotSameNotIso}  With notation as above,  suppose that there are distinct natural numbers $k, k'$ such that 
\[\mathcal T^{k}_{\alpha}(A_0) = \mathcal T^{k'}_{\alpha}(D_0)\,.\]
  Then the entropy of $(\mathcal T_{\alpha}, \Omega_{\alpha},\overline{ \mathcal B}_{\alpha}, \mu\,)$  differs from that of
$(\mathcal T_{1/2}, \Omega_{1/2}, \overline{\mathcal B}_{1/2}, \mu\,)\,$.
\end{Lem} 

\begin{proof}(Sketch)  If $k < k'$,  then we can produce a new system $(\Omega_{\alpha}', T'_{\alpha}, \overline{ \mathcal B}'_{\alpha}, \mu\,)$ by  inducing past one ``copy'' of $A_0\,$.   This system can be shown to be isomorphic to $(\mathcal T_{1/2}, \Omega_{1/2}, \overline{\mathcal B}_{1/2}, \mu\,)\,$.   But, by the Abramov formula \cite{A},   the induced system  has entropy differing from that of the full system by a multiplicative factor equal to the measure of $A_0\,$.   Thus, the entropy of $(\mathcal T_{\alpha}, \Omega_{\alpha},\overline{ \mathcal B}_{\alpha}, \mu\,)$ is less than that of the system of index $1/2\,$.   

Similarly,  if $k >k'\,$, we form a new system  by inducing past a copy of $D_0$ in  the index $1/2$ system.    This allows us to conclude that  the entropy of $(\mathcal T_{\alpha}, \Omega_{\alpha},\overline{ \mathcal B}_{\alpha}, \mu\,)$ is greater than that of the system of index $1/2\,$.   
\end{proof} 
 
The following is part of Theorem 2.9 of \cite{DKS}.
\begin{Lem}\label{staggeredOrbits}\rm{(Dajani {\em et al} \cite{DKS})}   For  odd $q$  and  $\alpha \in (\omega_0, 1/\lambda\,]$
there are distinct natural numbers $k, k'$ such that   $l_k = r_{k'}\,$ and that the $T_{\alpha}$-digits of 
$l_{k-1} = r_{k'-1}\,$ differ by one.   
\end{Lem} 

\begin{Cor}   For  odd $q$  and  $\alpha \in (\omega_0, 1/\lambda\,]\,$,  the maps 
 $T_{\alpha}$ and 
$ T_{1/2} $ have distinct entropy values.
\end{Cor} 

\begin{proof}   From  Lemma ~\ref{staggeredOrbits},  one can show that $\mathcal T^{k}_{\alpha}(A_0) = \mathcal T^{k'}_{\alpha}(D_0)\,$.   Lemma ~\ref{quiltNotSameNotIso} then applies.
\end{proof}

In fact,   Theorem 2.9 of \cite{DKS} shows that our  $k' = k+1\,$; thus,  from the proof of Lemma ~\ref{quiltNotSameNotIso}, we see that the entropy of $T_{\alpha}$ for $\alpha \in (\omega_0, 1/\lambda\,]$ is decreasing, confirming the results for the case $q=3$ of \cite{N}, see also \cite{LM}.

\end{document}